\documentclass{article}
\usepackage{amsfonts}

\usepackage{graphicx}
\usepackage{amsmath}

\newtheorem{theorem}{Theorem}
\newtheorem{lemma}[theorem]{Lemma}
\newenvironment{proof}[1][Proof]{\textbf{#1.} }{\ \rule{0.5em}{0.5em}}

\newtheorem{conjecture}[theorem]{Conjecture}

\begin{document}

\title{Tensor powers for non-simply laced Lie algebras\\
$B_2$-case }
\author{P P Kulish, V D Lyakhovsky and O V Postnova}

\author{P P Kulish$^1$, V D Lyakhovsky$^2$, and O V Postnova$^3$ \\
$^1$ Sankt-Petersburg Branch of \\
V A Steklov Mathematical Institute  RAS\\
$^2$,$^3$ Sankt-Petersburg State University,\\
High Energy Physics and Elementary Particles Department\\
$^1$e-mail: kulish@pdmi.ras.ru,\\
$^2$ lyakh1507@nm.ru\\
$^3$ olgapostnova@mail.ru
}

\maketitle

\begin{abstract}
We study the decomposition problem for tensor powers of $B_2$-fundamental modules.
To solve this problem singular weight technique and injection fan algorithms are applied. Properties of multiplicity coefficients are formulated in terms of multiplicity functions.
These functions are constructed showing explicitly the dependence of multiplicity coefficients on the highest weight coordinates and the tensor power parameter. It is thus possible to study general properties of multiplicity coefficients for powers of the fundamental $B_2$- modules.
\end{abstract}

\vspace{2 mm}

\section{Introduction}

\label{sec:Introduction}

Consider an analog of the Brauer centralizer algebras for the spinor groups
and define the subspaces of the tensor space $\left( \otimes ^{p}\mathbf{V}%
^{n}\right) $ on which the symmetric group $S_{k}$ and $\mathrm{Spin}(n)$
act as the dual pair (in a direct product form). Here $\mathbf{V}^{n}$ is
the fundamental representation of $\mathrm{Spin}(n)$. Namely the
centralizer algebra of the orthogonal group in $\left( \otimes ^{p}\mathbf{V}%
^{n}\right) $ is generated by the symmetric group $S_{k}$ and the
contractions and the immersions of the invariant form and is called the
Brauer centralizer algebras. To proceed further one needs the
list of $\mathrm{Spin}(n)$-irreducible subspaces in the decomposition $%
\left( \otimes ^{p}\mathbf{V}^{n}\right) =\sum_{\mu \in P}m_{\mu }^{p}%
\mathbf{V}^{\left( \mu \right) }$ ($P$ is the $\mathrm{Spin}(n)$ weight
space) and their multiplicities $m_{\mu }^{p}$. As far as we are interested
in an arbitrary power $p$ of the fundamental module $\mathbf{V}_{\mathrm{Spin%
}\left( n\right) }^{n}$ our main problem is to find multiplicities of
submodules in a form of \textit{multiplicity functions} $M\left( \mu
,p\right) $ explicitly depending on the corresponding highest weight $\mu $
and the power parameter $p$.

There are numerous combinatorial studies of the problem \cite
{KirillovReshetikhin1996, Kleber1996, Kleber1998, Chari2000} and also series
of works dealing with fermonic formulas, some of them based on crystal basis
approach \cite{HatayamaKuniba1998,
HatayamaKuniba2000,HatayamaKuniba2001,NaitoSagaki2006}. On this way
important general results were obtained \cite{LeducRam1997,KulishManojlovicNagy2010,Kumar2010}.
On the other hand practical
computations with the corresponding formulas are scarcely possible for all
but the simplest examples. In most of these studies the simply laced
algebras are considered and as a rule the multiplicities formulas are
connected with complicated path countings.

We must mention also an algorithm for tensor product decompositions proposed
in \cite{Klimyk1968} and improved in \cite{KlimykSchmudgen1997}.\ It is used
in our investigations.

Summing up, we are to find multiplicities $m_{\mu }^{p}$ in the
decomposition $\left( \otimes ^{p}\mathbf{V}^{n}\right) =\sum_{\mu \in
P}m_{\mu }^{p}\mathbf{V}^{\left( \mu \right) }$ as a function of $\mu $\ and
$p$\ . To solve this problem we propose an algorithm based on singular
weights properties \cite{Feigin1986} and the injection fan technique \cite
{IlyinKulishLyakhovsky2009,LyakhNaz2011}. We study the multiplicities $%
m_{\mu }^{p}$ formulated in terms of multiplicity functions $M_{\frak{g}%
}\left( \mu ,p\right) $. The latter have the weight space $\mathcal{L}P$ for
the domain of definition. On the sublattice $P^{++}$ of dominant weights the
multiplicity function gives us the desired multiplicities, $M\left( \mu
,p\right) |_{\mu \in P^{++}}=m_{\mu }^{p}$. In this paper we shall show how
to adopt these tools to non-simply laced algebras and shall demonstrate how
they work by studying the tensor powers $\left( L_{B_{n}}^{\omega
_{i}}\right) ^{\otimes ^{p}}$ of the fundamental module $L_{B_{2}}^{\omega
_{i}}$ of $B_{2}$.

\section{Basic definitions and relations.}

$\frak{g}$ -- simple Lie algebra of the series $B_{n}$ ,

$L^{\mu }$\ \ -- the integrable module of $\frak{g}$ with the highest weight
$\mu $\ ;

$r$ -- the rank of the algebra $\frak{g}$ ;

$\Delta $ -- the root system; $\Delta ^{+}$ -- the positive root system for $%
\frak{g}$ ;

$\mathcal{N}^{\mu }$ -- the weight diagram of $L^{\mu }$ ;

$W$ -- the Weyl group;

$C^{\left( 0\right) }$ -- the fundamental Weyl chamber, $\overline{C^{\left(
0\right) }}$ -- its closure;

$\rho $\ -- the Weyl vector;

$\epsilon \left( w\right) :=\det \left( w\right) $ , $w \in W$;

$\alpha _{i}$ -- the $i$-th simple root for $\frak{g}$ ; $i=0,\ldots ,r$ ;

$\omega _{i}$ -- the $i$-th fundamental weight for $\frak{g}$ ; $i=0,\ldots
,r$ ;

$L_{\frak{g}}^{\omega _{i}}$ -- the $i$-th fundamental module;

$\left\{ e^{i}\right\} _{\mid i=1,\ldots ,r}$ -- the natural Euclidean basis
of the weight space (the $e$-basis), $\left\{ v_{i}\right\} $ -- the
coordinates of a weight in the $e$-basis;

$P$ -- the weight lattice $\mathcal{L}P$ -- the weight space;

$Q$ -- the root lattice;

$\mathcal{E}$\ -- the group algebra of the group $P$ ;

$\Psi ^{\left( \mu \right) }:=\sum\limits_{w\in W}\epsilon (w)e^{w\circ (\mu
+\rho )-\rho }$ -- the singular element for the $\frak{g}$-module $L^{\mu }$;

$\widehat{\Psi ^{\left( \mu \right) }}$ -- the set of singular weights $\psi
\in P$ for the module $L^{\mu }$ with the coordinates $\left( \psi ,\epsilon
\left( w\left( \psi \right) \right) \right) \mid _{\psi =w\left( \psi
\right) \circ (\mu +\rho )-\rho }$;

$\mathrm{ch}\left( L^{\mu }\right) $ -- the formal character of $L^{\mu }$ ;

$\mathrm{ch}\left( L^{\mu }\right) =\frac{\sum_{w\in W}\epsilon (w)e^{w\circ
(\mu +\rho )-\rho }}{\prod_{\alpha \in \Delta ^{+}}\left( 1-e^{-\alpha
}\right) }=\frac{\Psi ^{\left( \mu \right) }}{\Psi ^{\left( 0\right) }}$ --
the Weyl formula;

$R:=\prod_{\alpha \in \Delta ^{+}}\left( 1-e^{-\alpha }\right) =\Psi
^{\left( 0\right) }$ -- the denominator;

$M_{\frak{g}}^{\omega _{i}}\left( \mu ,p\right) $ -- the multiplicity
function corresponding to the decomposition $\left( L_{\frak{g}}^{\omega
_{i}}\right) ^{\otimes ^{p}}=\sum m_{\mu }^{\left( i\right) p}L_{\frak{g}%
}^{\mu }$ , $M_{\frak{g}}^{\omega _{i}}\left( \mu ,p\right) |_{\mu \in
P^{++}}=m_{\mu }^{\left( i\right) p}.$

\section{Some useful properties.}

\begin{lemma}
The projection $\Psi _{\downarrow \frak{g}}^{\left( \nu ,\xi \right) }$ of
the singular element $\Psi _{\frak{g}\oplus \frak{g}}^{\left( \nu ,\xi
\right) }$ for the irreducible representation $L^{\left( \nu _{1},\ldots
,\nu _{r}\right) }\otimes L^{\left( \xi _{1},\ldots ,\xi _{r}\right) }$ of
the direct sum $\frak{g}\oplus \frak{g}$ on the weight space of the diagonal
subalgebra $\frak{g}\rightarrow \frak{g}\oplus \frak{g}$. is equal to the
product
\begin{equation*}
\Psi _{\downarrow \frak{g}}^{\left( \nu ,\xi \right) }=\Psi _{\frak{g}}^{\nu
}\Psi _{\frak{g}}^{\xi },
\end{equation*}
Let $\left\{ \psi _{k}^{\left( \nu \right) }\mid \psi _{k}^{\left( \nu
\right) }\in \Psi ^{\left( \nu \right) },k=1,\ldots ,\#W\right\} _{\mid }$
and $\left\{ \psi _{p}^{\left( \xi \right) }\mid \psi _{p}^{\left( \xi
\right) }\in \Psi ^{\left( \xi \right) },p=1,\ldots ,\#W\right\} $ be the
sets of singular weights for the modules $L^{\left( \nu _{1},\ldots ,\nu
_{r}\right) }$ and $L^{\left( \xi _{1},\ldots ,\xi _{r}\right) }$
correspondingly then the set $\widehat{\Psi _{\downarrow \frak{g}}^{\left(
\nu ,\xi \right) }}$ consists of the weights $\left\{ \psi _{k}^{\left( \nu
\right) }+\psi _{p}^{\left( \xi \right) },\epsilon \left( w\left( \psi
_{k}^{\left( \nu \right) }\right) \right) \epsilon \left( w\left( \psi
_{p}^{\left( \xi \right) }\right) \right) \right\} .$
\end{lemma}

\begin{proof}
Let $\left\{ e^{1},e^{2},\ldots ,e^{r},e^{r+1},\ldots e^{2r}\right\} $ be
the weight space basis for $\mathcal{L}P\left( \frak{g}\oplus \frak{g}%
\right) $, $\left( \nu _{1},\ldots ,\nu _{r},\xi _{1},\ldots ,\xi
_{r}\right) $ -- the coordinates for the highest weight $\left( \nu ,\xi
\right) $ naturally belonging to the space $\mathcal{L}P\left( \frak{g}%
\oplus \frak{g}\right) $. The weights $v^{\left( \nu \right) }\in \mathcal{N}%
^{\nu }$ , $u^{\left( \xi \right) }\in \mathcal{N}^{\xi }$\ and the singular
vectors $\psi _{k}^{\left( \nu \right) }$ and $\psi _{p}^{\left( \xi \right)
}$ also are lifted to the space $\mathcal{L}P\left( \frak{g}\oplus \frak{g}%
\right) $
\begin{eqnarray*}
lv_{a}^{\left( \nu \right) } &\Rightarrow &\left( \nu _{a1}^{\left( \nu
\right) },\ldots ,\nu _{ar}^{\left( \nu \right) },0,\ldots ,0\right)
;lu_{b}^{\left( \xi \right) }\Rightarrow \left( 0,\ldots ,0,u_{b1}^{\left(
\eta \right) },\ldots ,u_{br}^{\left( \eta \right) }\right)  \\
l\psi _{k}^{\left( \nu \right) } &\Rightarrow &\left( \psi _{k1}^{\left( \nu
\right) },\ldots ,\psi _{kr}^{\left( \nu \right) },0,\ldots ,0\right) ;l\psi
_{p}^{\left( \xi \right) }\Rightarrow \left( 0,\ldots ,0,\psi _{p1}^{\left(
\xi \right) },\ldots ,\psi _{pr}^{\left( \xi \right) }\right)
\end{eqnarray*}
The set $\left\{ lv_{a}+lu_{b}\mid a=1,\ldots ,\dim \left( L^{\nu }\right)
,b=1,\ldots ,\dim \left( L^{\xi }\right) \right\} $ forms the weight diagram
$\mathcal{N}_{\frak{g}\oplus \frak{g}}^{\left( \nu ,\xi \right) }$ of $L_{%
\frak{g}\oplus \frak{g}}^{\left( \nu ,\xi \right) }$ . As far as for the
Weyl group $W_{\frak{g}\oplus \frak{g}}$ we have $W_{\frak{g}\oplus \frak{g}%
}=W\times W$ and the Weyl vector is $\rho _{\frak{g}\oplus \frak{g}}=\left(
\rho ,\rho \right) $\ , the set of singular weights $\widehat{\Psi _{\frak{g}%
\oplus \frak{g}}^{\left( \nu ,\xi \right) }}$ is formed by the vectors whose
first $2r$ coordinates are $\left\{ l\psi _{k}^{\left( \nu \right) }+l\psi
_{p}^{\left( \xi \right) }\mid k,p=1,\ldots ,\#W\right\} $ and the last one
is equal to the product $\epsilon \left( w\left( \psi _{k}^{\left( \nu
\right) }\right) \right) \epsilon \left( w\left( \psi _{p}^{\left( \xi
\right) }\right) \right) $ :
\begin{equation*}
\widehat{\Psi _{\frak{g}\oplus \frak{g}}^{\left( \nu ,\xi \right) }}=\left\{
l\psi _{k}^{\left( \nu \right) }+l\psi _{p}^{\left( \xi \right) },\epsilon
\left( w\left( \psi _{k}^{\left( \nu \right) }\right) \right) \epsilon
\left( w\left( \psi _{p}^{\left( \xi \right) }\right) \right) \mid
k,p=1,\ldots ,\#W\right\} .
\end{equation*}
The vector $\left( c_{1},\ldots ,c_{r},c_{r+1},\ldots ,c_{2r}\right) \in
\mathcal{L}P\left( \frak{g}\oplus \frak{g}\right) $ being projected to the
diagonal subalgebra weight space $\mathcal{L}P$ in the basis $\left\{ \left(
e^{1}+e^{r+1}\right) /2,\ldots ,\left( e^{r}+e^{2r}\right) /2\right\} $ has
the coordinates $\left( \left( c_{1}+c_{r+1}\right) ,\ldots ,\left(
c_{r}+c_{2r}\right) \right) $ . The latter means that
\begin{eqnarray*}
\widehat{\Psi _{\downarrow \frak{g}}^{\left( \nu ,\xi \right) }} &=&\left\{
\psi _{k}^{\left( \nu \right) }+\psi _{p}^{\left( \xi \right) },\epsilon
\left( w\left( \psi _{k}^{\left( \nu \right) }\right) \right) \epsilon
\left( w\left( \psi _{p}^{\left( \xi \right) }\right) \right) \right\} , \\
k,p &=&1,\ldots ,\#W.
\end{eqnarray*}
\end{proof}

Q.E.D.

One of the main tools to study the decomposition properties is the \textit{%
injection fan }$\Gamma _{\frak{g}_{\mathrm{diag}}\rightarrow \frak{g}\oplus
\frak{g}}$. \cite{IlyinKulishLyakhovsky2009,LyakhNaz2011} . To use this
instrument we consider the decomposition of tensor products as a special
case of branching. The latter corresponds to the injection of the diagonal
subalgebra into the direct sum: $\frak{g}_{\mathrm{diag}}\rightarrow \frak{g}%
\oplus \frak{g}$.

\begin{lemma}
The vectors of the injection fan $\Gamma $ for $\frak{g}_{\mathrm{diag}%
}\rightarrow \frak{g}\oplus \frak{g}$ consists of the opposites to the
singular weights of the trivial module $L_{\frak{g}}^{\left( 0\right) }$
\begin{equation*}
\Gamma _{\frak{g}_{\mathrm{diag}}\rightarrow \frak{g}\oplus \frak{g}%
}=-S\circ \widehat{\Psi ^{\left( 0\right) }}\backslash \left( 0,\ldots
,0\right) ,
\end{equation*}
(here S is the full reflection).
\end{lemma}

\begin{proof}
According to the definition \cite{IlyinKulishLyakhovsky2009} the vectors $%
\gamma $ of the fan $\Gamma _{\frak{g}_{\mathrm{diag}}\rightarrow \frak{g}%
\oplus \frak{g}}$ are fixed by the relation $1-\prod_{\alpha \in \left(
\Delta _{\frak{g}\oplus \frak{g}\downarrow \frak{g}_{\mathrm{diag}%
}}^{+}\right) }\left( 1-e^{-\alpha }\right) ^{\mathrm{mult}_{\frak{g}\oplus
\frak{g}}\mathrm{\left( \alpha \right) -{mult}}_{\frak{g}_{\mathrm{diag}}}%
\mathrm{\left( \alpha \right) }}=\sum_{\gamma \in \Gamma _{\frak{g}_{\mathrm{%
diag}}\rightarrow \frak{g}\oplus \frak{g}}}s\left( \gamma \right) e^{-\gamma
}$. The projections of the\ $\frak{g}\oplus \frak{g}$ -roots to the diagonal
subalgebra obviously reproduce the set $\left\{ \alpha _{i}\right\}
_{i=0,\ldots ,r}$ in $\mathcal{L}P\left( \frak{g}\right) $ :
\begin{equation*}
\left.
\begin{array}{c}
\left( \alpha _{i},0\right) _{\downarrow \frak{g}} \\
\left( 0,\alpha _{i}\right) _{\downarrow \frak{g}}
\end{array}
\right\} =\alpha _{i}.
\end{equation*}
Thus $\mathrm{mult}_{\frak{g}}\left( \alpha \right) =2$ while $\mathrm{mult}%
_{\frak{g}\oplus \frak{g}}\left( \alpha \right) =1$ and we have
\begin{eqnarray*}
\sum_{\gamma \in \Gamma _{\frak{g}_{\mathrm{diag}}\rightarrow \frak{g}\oplus
\frak{g}}}s\left( \gamma \right) e^{-\gamma } &=&1-\prod_{\alpha \in \left(
\Delta _{\frak{g}\oplus \frak{g}\downarrow \frak{g}}^{+}\right) }\left(
1-e^{-\alpha }\right) ^{\mathrm{mult}_{\frak{g}\oplus \frak{g}}\left( \alpha
\right) -\mathrm{mult}_{\frak{g}}\left( \alpha \right) }= \\
&=&1-\prod_{\alpha \in \left( \Delta ^{+}\right) }\left( 1-e^{-\alpha
}\right) .
\end{eqnarray*}
Q.E.D.
\end{proof}

\begin{lemma}
The singular element $\Psi _{\frak{g}}^{\left( \xi \right) }$ for the module
$L^{\mu }\otimes L^{\nu }$ can be presented in two equivalent forms:
\begin{equation}
\mathrm{ch}\left( L_{\frak{g}}^{\mu }\right) \Psi _{\frak{g}}^{\nu }=\Psi _{\frak{g}%
}^{\mu }\mathrm{ch}\left( L_{\frak{g}}^{\nu }\right) .
\end{equation}
\end{lemma}

\begin{proof}
In the Weyl formula for $L^{\mu }\otimes L^{\nu }$,
\begin{equation*}
\mathrm{ch}\left( L^{\mu }\otimes L^{\nu }\right) _{\downarrow P_{\frak{g}%
}^{+}}=\sum_{\xi \in P_{\frak{g}}^{++}}m_{\xi }^{\mu \nu }\mathrm{ch}\left(
L^{\xi }\right) ,
\end{equation*}
apply the result of Lemma 1:
\begin{equation*}
\left( \frac{\Psi _{\frak{g}\oplus \frak{g}}^{\left( \nu ,\xi \right) }}{%
\Psi _{\frak{g}\oplus \frak{g}}^{0}}\right) _{\downarrow P_{\frak{g}}}=\frac{%
\Psi _{\frak{g}}^{\mu }\Psi _{\frak{g}}^{\nu }}{\Psi _{\frak{g}}^{0}\Psi _{%
\frak{g}}^{0}}=\sum_{\xi \in P_{\frak{g}}^{++}}m_{\xi }^{\mu \nu }\frac{\Psi
_{\frak{g}}^{\xi }}{\Psi _{\frak{g}}^{0}}.
\end{equation*}
\begin{equation*}
\left( \left( \Psi _{\frak{g}}^{0}\right) ^{-1}\Psi _{\frak{g}}^{\mu
}\right) \Psi _{\frak{g}}^{\nu }=\Psi _{\frak{g}}^{\mu }\left( \left( \Psi _{%
\frak{g}}^{0}\right) ^{-1}\Psi _{\frak{g}}^{\nu }\right) =\sum_{\xi \in P_{%
\frak{g}}^{++}}m_{\xi }^{\mu \nu }\Psi _{\frak{g}}^{\xi }.
\end{equation*}
Thus we have
\begin{equation}
\sum_{\xi \in P_{\frak{g}}^{++}}m_{\xi }^{\mu \nu }\Psi _{\frak{g}}^{\xi
}=\mathrm{ch}\left( L_{\frak{g}}^{\mu }\right) \Psi _{\frak{g}}^{\nu }=\Psi _{\frak{g}%
}^{\mu }\mathrm{ch}\left( L_{\frak{g}}^{\nu }\right) .
\label{N-recursion}
\end{equation}
Q.E.D.
\end{proof}

Now put $\mu =\omega $, $\nu =\left( p-1\right) \omega $
\begin{equation}
\mathrm{ch}\left( L^{\left( \omega \right) }\right) \Psi ^{\left(
\otimes ^{\left( p-1\right) }\omega \right) }=\sum_{\xi \in
P}M_{\frak{g}}^{\omega _{1}}\left( \xi ,p\right) \Psi ^{\left( \xi
\right) }, \label{N-rec property}
\end{equation}
$M_{\frak{g}}^{\omega _{1}}\left( \xi ,p\right) $ defines the singular
element $\Psi ^{\left( \otimes ^{p}\omega \right) }$. On the other hand
these equations can be considered as a system of recurrent relations for the
multiplicity function $M_{\frak{g}}^{\omega _{1}}\left( \xi ,p\right) $ .

\begin{conjecture}
Let $\frak{g}=B_{2}$, $L^{\mu }$\ and $L^{\omega _{1}}=L^{vect}$ be the
highest weight modules with the highest weights $\mu $ and $\omega _{1}$
(the first fundamental weight). Then the tensor product decomposition $%
\left( L^{\mu }\otimes L^{vect}\right) \downarrow _{\frak{g}_{\mathrm{diag}%
}}=\oplus _{\gamma }L^{\gamma }$ is multiplicity free.
\end{conjecture}

\begin{proof}
According to Lemma 1 the projected singular element for a $\frak{g}\oplus
\frak{g}$ -module $L^{\mu }\otimes L^{vect}$ is $\Psi _{\frak{g}\oplus \frak{%
g}\downarrow \frak{g}_{\mathrm{diag}}}^{\left( \mu ,vect\right) }=\Psi _{%
\frak{g}}^{\mu }\Psi _{\frak{g}}^{vect}$ and the set of singular weights is
\begin{eqnarray*}
\widehat{\Psi _{\downarrow \frak{g}_{\mathrm{diag}}}^{\left( \mu
,vect\right) }} &=&\left\{ \psi _{k}^{\left( \mu \right) }+\psi _{p}^{\left(
vect\right) },\epsilon \left( w\left( \psi _{k}^{\left( \mu \right) }\right)
\right) \epsilon \left( w\left( \psi _{p}^{\left( vect\right) }\right)
\right) \right\} , \\
k,p &=&1,\ldots ,\#W.
\end{eqnarray*}
Suppose $\mu =\left( \mu _{1},\mu _{2}\right) $ is greater than $\omega
_{1}=\left( 1,0\right) $ and $\mu _{1}>\mu _{2}\geq 1$. The singular weights
of $L^{\mu }$ in the fundamental chamber and its nearest neighbours are
\begin{eqnarray*}
\left\{ \psi _{s}^{\left( \mu \right) }\right\}  &=&\left( \mu _{1},\mu
_{2},\left( +1\right) \right) ,\left( \mu _{1},-\mu _{2}-1,\left( -1\right)
\right) ,\left( \mu _{2}-1,\mu _{1}+1,\left( -1\right) \right)  \\
s &=&1,2,3
\end{eqnarray*}
They will give rise to the 24 weights of the type
\begin{eqnarray*}
\left\{ \psi _{s}^{\left( \mu \right) }+\psi _{p}^{\left( vect\right)
}\right\}  &=&\left( \mu _{1},\mu _{2},\left( +1\right) \right) +\psi
_{p}^{\left( vect\right) }, \\
&&\left( \mu _{1},-\mu _{2}-1,\left( -1\right) \right) +\psi _{p}^{\left(
vect\right) }, \\
&&\left( \mu _{2}-1,\mu _{1}+1,\left( -1\right) \right) +\psi _{p}^{\left(
vect\right) } \\
s &=&1,2,3;p=1,\ldots ,\#W.
\end{eqnarray*}
Applying successively the fan $\Gamma _{\frak{g}_{\mathrm{diag}}\rightarrow
\frak{g}\oplus \frak{g}}$ to the set $\widehat{\Psi _{\downarrow \frak{g}_{%
\mathrm{diag}}}^{\left( \mu ,vect\right) }}$ in the three selected chambers
(starting with the highest weight $\left( \mu _{1}+1,\mu _{2},\left(
+1\right) \right) $) we find the weights:
\begin{eqnarray*}
&&\left( \mu _{1},\mu _{2},\left( +1\right) \right) ,\left( \mu _{1}\pm
1,\mu _{2},\left( +1\right) \right) ,\left( \mu _{1},\mu _{2}\pm 1,\left(
+1\right) \right) , \\
&&\left( \mu _{1},-\mu _{2}-1,\left( -1\right) \right) ,\left( \mu _{1}\pm
1,-\mu _{2}-1,\left( -1\right) \right) ,\left( \mu _{1},-\mu _{2},\left(
-1\right) \right) ,\left( \mu _{1},-\mu _{2}-2,\left( -1\right) \right) , \\
&&\left( \mu _{2}-1,\mu _{1}+1,\left( -1\right) \right) ,\left( \mu _{2},\mu
_{1}+1,\left( -1\right) \right) ,\left( \mu _{2}-2,\mu _{1}+1,\left(
-1\right) \right) , \\
&&\left( \mu _{2}-1,\mu _{1}+2,\left( -1\right) \right) ,\left( \mu
_{2}-1,\mu _{1},\left( -1\right) \right) .
\end{eqnarray*}
Only the first 5 weights are in $\overline{C^{\left( 0\right) }}$ . They are
the highest singular weights for $\left( L^{\mu }\otimes L^{vect}\right)
_{\downarrow \frak{g}_{\mathrm{diag}}}$ and we have the decomposition:
\begin{equation*}
\left( L^{\mu }\otimes L^{vect}\right) _{\downarrow \frak{g}_{\mathrm{diag}%
}}=L^{\mu }\oplus L^{\left( \mu _{1}+1,\mu _{2}\right) }\oplus L^{\left( \mu
_{1}-1,\mu _{2}\right) }\oplus L^{\left( \mu _{1},\mu _{2}+1\right) }\oplus
L^{\left( \mu _{1},\mu _{2}-1\right) }.
\end{equation*}
There are two special cases where the highest weights are on the borders of $%
\overline{C^{\left( 0\right) }}$. For $\mu =n\omega _{1}=\left( n,0\right) $
and for $\mu =n\omega _{2}=\left( n/2,n/2\right) $ the same algorithm gives:
\begin{equation*}
\left( L^{\mu }\otimes L^{vect}\right) _{\downarrow \frak{g}_{\mathrm{diag}%
}}=L^{\left( \mu _{1}+1,0\right) }\oplus L^{\left( \mu _{1}-1,0\right)
}\oplus L^{\left( \mu _{1},\mu _{2}+1\right) },
\end{equation*}

\begin{equation*}
\left( L^{\mu }\otimes L^{vect}\right) _{\downarrow \frak{g}_{\mathrm{diag}%
}}=L^{\mu }\oplus L^{\left( \mu _{1}+1,\mu _{2}\right) }\oplus L^{\left( \mu
_{1},\mu _{2}-1\right) }.
\end{equation*}
correspondingly. Q.E.D.
\end{proof}

\section{Singular elements and fans. $B_{2}$-case}

For $\frak{g}=B_{2}$ , $r=2$ the simple roots in $e$-basis are $\alpha
_{1}=e_{1}-e_{2}, \alpha _{2}=e_{2}$, the fundamental weights are $\omega
_{2}=\frac{1}{2}\left( e_{1}+e_{2}\right) ,\omega _{1}=e_{1}$ and the
fundamental modules -- $L^{\omega _{2}}$(spinor) and $\dim L^{\omega
_{2}}=4,L^{\omega _{1}}$(vector), $\dim L^{\omega _{1}}=5$.

Consider the modules $(L^{\omega _{i}})^{\otimes p}|_{p\in Z_{+},i=1,2}$ and
the decompositions $(L^{\omega _{i}})^{\otimes p}=\sum_{\nu }m_{\nu
}^{\left( i\right) p}L^{\nu }.$

Our aim is to find multiplicities $m_{\nu }^{\left( i\right) p }$ as
functions of $\nu $\ and $p$. To solve the problem we propose to use the
\textit{singular elements formalism }\cite{Feigin1986}, the polynomial
dependence property that is a consequence of the relation (\ref{N-rec
property}) and the injection fan technique \cite{IlyinKulishLyakhovsky2009}
\cite{LyakhNaz2011}.

\subsection{Constructing the fan $\Gamma _{B_{2}^{\mathrm{diag}}\rightarrow
\oplus ^{p}B_{2}}$}

Consider the injection $B_{2}^{\mathrm{diag}}\rightarrow \oplus ^{p}B_{2}$.
The fan $\Gamma _{B_{2}^{\mathrm{diag}}\rightarrow \oplus ^{p}B_{2}}\equiv
\Gamma _{p}$ is the $\left( p-1\right) $-th tensor power of the trivial
module singular element that is of $\Psi _{B_{2}}^{\left( 0\right) }$.

\begin{conjecture}
Place the origin of the space $\mathcal{L}P$ at the end of the lowest weight
vector of the fan. The structure of the fan $\Gamma _{p}$ is as follows:

\begin{enumerate}
\item  Along the line $p\alpha _{1}$ in the $k$-th root lattice point the
multiplicity is $\left( -1\right) ^{k}C_{p}^{k-1}$; $k=1,\ldots ,p+1.$

\item  Each weight with the coordinates $\left( k-1,1-k\right) $ is an
origin of the set $S_{\left( k-1,1-k\right) }$ of singular weights described
below.

\item  The set $S_{\left( k-1,1-k\right) }$ is composed of the tensor
product of the singular elements $\Psi _{x}^{\left[ 0\right] <}$\ and $\Psi
_{y}^{\left[ v\right] <}$of the trivial and vector lowest weight modules of
the algebras $A_{1}$ with the roots $x$ and $y$ correspondingly. The list of
modules for the set $S_{\left( k-1,1-k\right) }$ is completely defined by
its coordinates $\left( k-1,1-k\right) $:
\begin{equation*}
S_{\left( k-1,1-k\right) }=\left( \Psi _{e_{2}}^{\left[ v\right] <}\right)
^{\otimes \left( k-1\right) }\otimes \left( \Psi _{e_{1}}^{\left[ 0\right]
<}\right) ^{\otimes \left( k-1\right) }\otimes \left( \Psi _{e_{1}}^{\left[ v%
\right] <}\right) ^{\otimes \left( p-k+1\right) }\otimes \left( \Psi
_{e_{2}}^{\left[ 0\right] <}\right) ^{\otimes \left( p-k+1\right) }.
\end{equation*}
\end{enumerate}
\end{conjecture}

\begin{proof}
Let $\Gamma _{p}$ be the fan with the properties described above. Remember
that the set $S_{\left( k-1,1-k\right) }$ has itself the multiplicity $%
\left( -1\right) ^{k}C_{p}^{k-1}$. Multiply the fan $\Gamma _{p}$ by the
element $\Psi ^{\left( 0\right) }$, i.e. pass to the power $\left(
p+1\right) $. This means that the set $S_{\left( k-1,1-k\right) }$ will be
transformed to
\begin{equation*}
\left( -1\right) ^{k}C_{p}^{k-1}\left( \Psi _{e_{2}}^{\left[ v\right]
<}\right) ^{\otimes \left( k-1\right) }\otimes \left( \Psi _{e_{1}}^{\left[ 0%
\right] <}\right) ^{\otimes \left( k-1\right) }\otimes \left( \Psi _{e_{1}}^{%
\left[ v\right] <}\right) ^{\otimes \left( p-k+2\right) }\otimes \left( \Psi
_{e_{2}}^{\left[ 0\right] <}\right) ^{\otimes \left( p-k+2\right) }.
\end{equation*}
The set $S_{\left( k-2,2-k\right) }$ will become $S_{\left( k-1,1-k\right) }$
with the multiplicity $\left( -1\right) ^{k+1}C_{p}^{k-2}$. According to the
Pascal triangle property,
\begin{equation*}
\left( -1\right) ^{k}C_{p}^{k-1}+\left( -1\right) ^{k+1}C_{p}^{k-2}=\left(
-1\right) ^{k}C_{p+1}^{k-1},
\end{equation*}
the first set $S_{\left( 0,0\right) }$ takes the form
\begin{equation*}
\left( \Psi _{e_{1}}^{\left[ v\right] <}\right) ^{\otimes \left( p+1\right)
}\otimes \left( \Psi _{e_{2}}^{\left[ 0\right] <}\right) ^{\otimes \left(
p+1\right) },
\end{equation*}
while the last one becomes $S_{\left( p,p\right) }$,
\begin{equation*}
\left( -1\right) ^{p+2}\left( \Psi _{e_{2}}^{\left[ v\right] <}\right)
^{\otimes \left( p+1\right) }\otimes \left( \Psi _{e_{1}}^{\left[ 0\right]
<}\right) ^{\otimes \left( p+1\right) }.
\end{equation*}
Thus the structure of this product coinsides with the previously defined fan
$\Gamma _{q}$ with $q=p+1$. Q.E.D.
\end{proof}

An explicit expression for the fan $\Gamma _{p}$ is obtained by substituting
$\Psi _{e_{i}}^{\left[ 0\right] <}$ and $\Psi _{e_{i}}^{\left[ v\right] <}$
(as formal algebra elements) by their expressions in terms of a function
\begin{equation*}
\widehat{C}_{j}^{i}=\left\{
\begin{array}{c}
C_{j}^{i}\text{ \ for }0<i\leq j>0 \\
0\text{ \ \ \ \ \ \ \ \ \ \ \ \ otherwise}.
\end{array}
\right.
\end{equation*}
Finally we get
\begin{eqnarray}
\Gamma _{p} &=&\sum_{a,b}\gamma _{p}\left( a,b\right) e^{\left( a,b\right)
};\quad \left\{
\begin{array}{c}
a=k-1,\ldots ,3p-k-2, \\
b=1-k,\ldots ,p+k-2,
\end{array}
\right.  \notag \\
\gamma _{p}\left( a,b\right)
&=&\sum_{k=1}^{p}\sum_{l_{k}=1}^{k}\sum_{m_{k}=1}^{p-k+1}e^{\left(
a,b\right) }\left( -1\right) ^{k+a+b-2\left( l_{k}+m_{k}\right) }\times
\notag \\
&&\times \widehat{C}_{p-1}^{k-1}\widehat{C}_{k-1}^{l_{k}-1}\widehat{C}%
_{p-k}^{m_{k}-1}\widehat{C}_{p-k}^{b+k-3l_{k}+2}\widehat{C}%
_{k-1}^{a-k-3m_{k}+4},  \label{fan-so5}
\end{eqnarray}
Here the fan is a function of the parameter $p$ and the coordinates $\left(
a,b\right) $ of the highest weight. The zero point has the multiplicity $%
\left( -1\right) $. The result is that for any $\mu \in P$ the singular
weights diagram $\Psi \left( \sum_{\nu }m_{\nu }^{\left( i\right) p}L^{\nu
}\right) =\sum e^{\left( \mu \right) }\epsilon \left( w\left( \psi \left(
\mu \right) \right) \right) $ has the following fundamental property:
\begin{equation*}
\sum_{a,b}\gamma _{p}\left( a,b\right) \epsilon \left( w\left( \psi \left(
\mu +\left( a,b\right) \right) \right) \right) =0,
\end{equation*}
described by the fan $\Gamma _{p}$.

\subsection{Singular element for $\left( L^{\protect\omega _{2}}\right)
^{\otimes p}$ -- the spinor case}

Let us construct the singular element for the $p$-th power of the second
(spinor) fundamental module $L^{\omega _{2}}$ (coordinates of singular
weights here are half-integer and the $W$-invariant vector is $(0,0)$).

\begin{enumerate}
\item  Define the system $S_{k}$ consisting of blocks enumerated by a pair
of indices $\left( l_{k},m_{k}\right) $ where $l_{k}=1,\ldots ,k+1$ and $%
m_{k}=1,\ldots ,p-k+2$ and attached to the points $\left( \frac{p}{2}%
-k+1-4\left( m_{k}-1\right) ,\frac{p}{2}+k-1-4\left( l_{k}-1\right) \right) $%
. The multiplicities of these blocks are $\left( -1\right) ^{l_{k}+m_{k}-2}%
\widehat{C}_{k}^{l_{k}-1}\widehat{C}_{p-k+1}^{m_{k}-1}$

\item  Localize the systems $S_{k}$ along the line $(\frac{p}{2},\frac{p}{2}%
)-p\alpha _{1}$: the first system $S_{1}$ has the origin at the point $(%
\frac{p}{2},\frac{p}{2})$, the $k$-th -- at the point $\left( \frac{p}{2}%
-k+1,\frac{p}{2}+k-1\right) $, the last one, $S_{p+1}$, -- at $\left( -\frac{%
1}{2}p,\frac{3}{2}p\right) $. These systems have the multiplicities $\left(
-1\right) ^{k-1}C_{p}^{k-1}$; $k=1,\ldots ,p+1.$

\item  The numbers $\left( \frac{p}{2}-k+1-4\left( m_{k}-1\right) ,\frac{p}{2%
}+k-1-4\left( l_{k}-1\right) \right) $ are the coordinates of the upper
right corner of the $\left( l_{k},m_{k}\right) $-block. The blocks have the
structure dual to the structure of the system $S_{k}$ but the intervals in
the blocks are doubled. The weights in the block are enumerated by the
indices $\left( i_{k},j_{k}\right) $ where $j_{k}=1,\ldots ,k+1$ and $%
i_{k}=1,\ldots ,p-k+2$.

\item  Thus the $\left( l_{k},m_{k}\right) $-block in $S_{k}$ has the form:
\begin{equation*}
\sum_{i_{k}=1}^{p-k+2}\sum_{j_{k}=1}^{k+1}\left( -1\right) ^{i_{k}+j_{k}-2}%
\widehat{C}_{k}^{j_{k}-1}\widehat{C}_{p-k+1}^{i_{k}-1}e^{\left( \frac{p}{2}%
-k+1-4m_{k}-2j_{k}+7,\frac{p}{2}+k-1-4l_{k}-2i_{k}+5\right) }.
\end{equation*}

\item  Now the system $S_{k}$ can be composed:
\begin{eqnarray*}
&&\sum_{i_{k},m_{k}=1}^{p-k+2}\sum_{j_{k},l_{k}=1}^{k+1}\left( -1\right)
^{l_{k}+m_{k}+i_{k}+j_{k}-4}\times  \\
&&\times \widehat{C}_{k}^{l_{k}-1}\widehat{C}_{p-k+1}^{m_{k}-1}\widehat{C}%
_{k}^{j_{k}-1}\widehat{C}_{p-k+1}^{i_{k}-1}e^{\left( \frac{p}{2}%
-k+1-4m_{k}-2j_{k}+7,\frac{p}{2}+k-1-4l_{k}-2i_{k}+5\right) }.
\end{eqnarray*}

\item  Finally, the singular element $\Psi ^{\left( \left( \omega
_{2}\right) ^{\otimes p}\right) }\equiv \Psi ^{\left( \left( s\right)
^{\otimes p}\right) }$ is fixed as
\begin{eqnarray*}
\Psi ^{\left( \left( s\right) ^{\otimes p}\right) }
&=&\sum_{k=1}^{p+1}\sum_{i_{k},m_{k}=1}^{p-k+2}\sum_{j_{k},l_{k}=1}^{k+1}%
\left( -1\right) ^{l_{k}+m_{k}+i_{k}+j_{k}-4}\times  \\
&&\times \widehat{C}_{k}^{l_{k}-1}\widehat{C}_{p-k+1}^{m_{k}-1}\widehat{C}%
_{k}^{j_{k}-1}\widehat{C}_{p-k+1}^{i_{k}-1}\times  \\
&&\times e^{\left( \frac{p}{2}-k+1-4m_{k}-2j_{k}+7,\frac{p}{2}%
+k-1-4l_{k}-2i_{k}+5\right) }
\end{eqnarray*}
The singular multiplicities are functions of $p$ and $c,d$:
\begin{eqnarray*}
\psi ^{\left( \left( s\right) ^{\otimes p}\right) }\left( c,d\right)
&=&e^{\left( c,d\right)
}\sum_{k=1}^{p+1}\sum_{l_{k}=1}^{k}\sum_{m_{k}=1}^{p-k+2}\left( -1\right)
^{k+\frac{1}{2}\left( c-d\right) -\left( l_{k}+m_{k}\right) +1}\times  \\
&&\times \widehat{C}_{p}^{k-1}\widehat{C}_{k}^{l_{k}-1}\widehat{C}_{k}^{%
\frac{1}{2}\left( 4\left( 1-m_{k}\right) -k+c-\frac{1}{2}p+1\right) }\times
\\
&&\times \widehat{C}_{p-k+1}^{m_{k}-1}\widehat{C}_{k}^{\frac{1}{2}\left(
2-4m_{k}+k-d+\frac{1}{2}p+1\right) }.
\end{eqnarray*}
\end{enumerate}

\subsection{Singular element for $\left( L^{\protect\omega _{1}}\right)
^{\otimes p}$ -- the vector case}

The construction procedure in the vector case is analogous to that of the
spinor. It results in obtaining the expression:
\begin{eqnarray*}
\Psi ^{\left( \left( \omega _{1}\right) ^{\otimes p}\right) } &\equiv &\Psi
^{\left( \left( v\right) ^{\otimes p}\right)
}=\sum_{k=1}^{p+1}\sum_{j,n=0}^{k-1}\sum_{i,m=0}^{p-k+1}e^{\left(
2k+j+5m-4p-2,-2p+i-k+5n+1\right) }\times \\
&&\times \left( -1\right) ^{i+j+k+m+n-1}\widehat{C}_{p-k+1}^{m}\widehat{C}%
_{k-1}^{n}\widehat{C}_{p-k}^{j}\widehat{C}_{p}^{k-1}\widehat{C}_{p-k+1}^{i}.
\end{eqnarray*}
The corresponding singular multiplicities function depending on $p$ and $c,d$
are
\begin{eqnarray}
\psi ^{\left( \left( v\right) ^{\otimes p}\right) }\left( c,d\right)
&=&\sum_{k=1}^{p+1}\sum_{l_{k}=1}^{k}\sum_{m_{k}=1}^{p-k+2}e^{\left(
c,d\right) }\times  \notag \\
&&\times \left( -1\right) ^{k-d-c+p-4\left( l_{k}+m_{k}\right) +7}\times
\notag \\
&&\times \widehat{C}_{p}^{k-1}\widehat{C}_{k-1}^{l_{k}-1}\widehat{C}%
_{p-k+1}^{m_{k}-1}\widehat{C}_{p-k+1}^{-d+2k-5\left( l_{k}-1\right) -2}%
\widehat{C}_{k-1}^{p-c-2k-5\left( m_{k}-1\right) +2}.
\label{sing-el-vect-so5}
\end{eqnarray}

\subsection{Recursive procedure for the vector case}

To illustrate the recursive algorithm we perform calculations that must give
us the multiplicities $m_{\nu }^{\left( 1\right) p}$. The starting value is
always known -- this is the multiplicity of the highest weight $\nu =\left(
p,0\right) $ of $(L^{\omega _{1}})^{\otimes p}$ that equals $1$. Suppose we
have found the values of the multiplicity function $m_{\nu }^{\left(
1\right) p}$ for the 14 first weights:\\[1mm]

\begin{tabular}{||c||cccccc||}
\hline\hline
$c\setminus d$ & $-2$ & $-1$ & $0$ & $+1$ & $+2$ & $+3$ \\ \hline\hline
$p$ & $0$ & $-1$ & $+1$ & $0$ & $0$ & $0$ \\
$p-1$ & $1-p$ & $0$ & $0$ & $p-1$ & $0$ & $0$ \\
$p-2$ &  &  &  &  & $\frac{1}{2}p\left( p-3\right) $ & $0$ \\ \hline\hline
\end{tabular}\\[1mm]

Applying the fan $\Gamma _{p}$\ (\ref{fan-so5}) we find the multiplicity for
the next weight in the third line, it has the coordinates $\left(
p-2,1\right) $. The first line of the fan weights contributes the value
\begin{equation*}
\left( p-1\right) \times \frac{1}{2}p\left( p-3\right) ,
\end{equation*}
the second line --
\begin{equation*}
-\left( p-1\right) \left( p-2\right) \times \left( p-1\right)
\end{equation*}
and the third line --
\begin{eqnarray*}
&&\left( +\frac{1}{2}\left( p-1\right) \left( p-2\right) \left( p-3\right)
\right) \times \left( +1\right) + \\
&&+\left( \left( p-1\right) -\frac{1}{2}\left( p-1\right) \left( p-2\right)
\right) \times \left( -1\right)
\end{eqnarray*}
Now we must calculate the singular element contribution -- the value of the
singular weights function $\psi ^{\left( \left( v\right) ^{\otimes p}\right)
}\left( p-2,1\right) $ for the weight $\left( p-2,1\right) $. The ''singular
contribution'' is
\begin{equation*}
\psi ^{\left( \left( v\right) ^{\otimes p}\right) }\left( p-2,1\right)
=p\left( p-1\right) .
\end{equation*}
The multiplicity $m_{\left( p-2,1\right) }^{\left( 1\right) p}$ is the sum:
\begin{eqnarray*}
m_{\left( p-2,1\right) }^{\left( 1\right) p} &=&\left( p-1\right) \left(
\begin{array}{c}
\frac{1}{2}p\left( p-3\right) -\left( p-1\right) \left( p-2\right) +p \\
+\frac{1}{2}\left( p-2\right) +\frac{1}{2}\left( p-2\right) \left(
p-3\right) -1
\end{array}
\right) \\
&=&\frac{1}{2}\left( p-1\right) \left( p-2\right) .
\end{eqnarray*}

Notice that here we consider the line $\nu =\left( p-2,1\right) $ in the
space $P\times \mathbf{R}^{1}$. As a result we obtain polynomials
characterizing the $p$-dependence of the multiplicity for a fixed distance
between the highest weight and the weight $\nu $.

This example shows that the tools elaborated above (the injection fan and
singular elements) are effective in solving the reduction problem for tensor
power modules $(L^{\omega _{i}})^{\otimes p}$.

\section{Alternative approach.}

According to Lemma 3 the multiplicity coefficients have additional
recurrence properties generated by the fundamental module weights system $%
N\left( L^{\left( \omega _{i}\right) }\right) $ :

\begin{equation}
\sum_{\mu \in P^{++}}m_{\mu }^{\left( i\right) p}\Psi ^{\left( \mu
\right) }=\mathrm{ch}\left( L^{\left( \omega _{i}\right) }\right)
\Psi ^{\left( \otimes ^{\left( p-1\right) }\omega _{i}\right) }.
\label{N-recurrence-fund}
\end{equation}
This relation can be decomposed using the multiplicity functions (defined on
$P$),
\begin{equation}
\sum_{\mu \in P}M^{\omega _{i}}\left( \mu ,p\right) e^{\mu
}=\mathrm{ch}\left( L^{\left( \omega _{i}\right) }\right) \Psi
^{\left( \otimes ^{\left( p-1\right) }\omega _{i}\right) },
\end{equation}
remember that $M^{\omega _{i}}\left( \mu ,p\right) \mid _{\mu \in \overline{%
C^{\left( 0\right) }}}=m_{\mu }^{\left( i\right) p}$ . Thus instead of the
highest weight search for each singular element $\Psi ^{\left( \mu \right) }$
we use their anti-symmetry properties. This leads to the recurrent relation:
\begin{equation}
M^{\omega _{i}}\left( \mu ,p\right) =\sum_{\zeta \in N\left( L^{\left(
\omega _{i}\right) }\right) }n_{\zeta }\left( L^{\left( \omega _{i}\right)
}\right) M^{\omega _{i}}\left( \mu -\zeta ,p-1\right) ,
\label{N-rec-rel-omega-i}
\end{equation}
where $n_{\zeta }\left( L^{\left( \omega _{i}\right) }\right) =\mathrm{mult}%
_{L^{\left( \omega _{i}\right) }}\left( \zeta \right) $. Obviously such
relations are especially useful when $\dim \left( L^{\left( \omega
_{i}\right) }\right) $ is small, thus for our needs (when the module $%
L^{\left( \omega _{i}\right) }$ is fundamental with the trivial weights
multiplicities $n_{\zeta }\left( L^{\left( \omega _{i}\right) }\right) =1$)
the obtained recurrence must be effective.

Formula (\ref{N-rec-rel-omega-i}) tells us what happens when we pass from
the $(p-1)$-th power to the $p$-th. In particular for the spinor module $%
L^{\omega _{2}}$ to find the value of $M^{\omega _{2}}\left( \mu
,p\right) $ the coordinates must be shifted by the vectors of
$\mathcal{N}\left( L^{\left( \omega _{2}\right) }\right) $
diagram:
\begin{equation}
M^{\omega _{2}}\left( \mu ,p\right) =\sum_{\zeta =N\left( L^{\left( \omega
_{2}\right) }\right) }M^{\omega _{2}}\left( \mu -\zeta ,p-1\right) .
\label{main rec rel}
\end{equation}
In the recurrence starting point the value of the multiplicity function is
known $M^{\omega _{2}}\left( p\omega _{2},p\right) =1$ and for all $\nu
>p\omega _{2}$ it has zero values. In the natural coordinates this means:
\begin{equation*}
M^{\omega _{2}}\left( \left( a,b\right) ,p\right) =\sum_{\lambda =\left(
a,b\right) -\left\{ \left( \frac{1}{2},\frac{1}{2}\right) \left( -\frac{1}{2}%
,\frac{1}{2}\right) \left( \frac{1}{2},-\frac{1}{2}\right) \left( -\frac{1}{2%
},-\frac{1}{2}\right) \right\} }M^{\omega _{2}}\left( \lambda ,p-1\right) .
\end{equation*}

For the vector module $L^{\omega _{1}}$ and its tensor powers we can
construct similar relations:
\begin{eqnarray}
M^{\omega _{1}}\left( \left( a,b\right) ,p\right) &=&\sum_{\zeta =N\left(
L^{\left( \omega _{1}\right) }\right) }M^{\omega _{1}}\left( \left(
a,b\right) -\zeta ,p-1\right) =  \notag \\
&&\sum_{\lambda =\left( a,b\right) -\left\{ \left( 1,0\right) \left(
0,1\right) \left( -1,0\right) \left( 0,-1\right) ,\left( 0,0\right) \right\}
}M^{\omega _{1}}\left( \lambda ,p-1\right) ,
\end{eqnarray}
with the similar boundary condition
\begin{equation}
M^{\omega _{1}}\left( \left( p,0\right) ,p\right) =1.
\end{equation}

The obtained recurrence relations indicate an important property of $%
M^{\omega _{i}}\left( \mu ,p\right) $:

\begin{conjecture}
The multiplicity function $M^{\omega _{i}}\left( \mu ,p\right) $
is a polynomial on $p$ over $\mathbf{Q}$ (rational numbers).
\end{conjecture}

Notice that when $\mu $ belongs to the correlated (different!) boundaries of
the area where the function is nontrivial the values $M^{\omega _{i}}\left(
\mu ,p\right) $ for $i=1$ and $i=2$ coinside.

\begin{conjecture}
The multiplicities $M^{\omega _{1}}\left( \mu ,p\right) $ of the ''upper
diagonal'' highest weights ($\mu =\left( p,0\right) -n\alpha _{1}$) for $%
(L^{\omega _{1}})^{\otimes p}$ coinside with the multiplicities $M^{\omega
_{2}}\left( \mu ,p\right) $ of the ''upper line'' highest weights ($\mu
=\left( \frac{p}{2},\frac{p}{2}\right) -n\alpha _{2}$) for $(L^{\omega
_{2}})^{\otimes p}$.
\end{conjecture}

On the left boundary of the Weyl chamber $\overline{C^{\left( 0\right) }}$
the multiplicities $M^{\omega _{1}}\left( \mu ,p\right) $ are subject to the
presence of the ''reflected'' singular weights in the left adjacent Weyl
chamber. This observation is important because for an analogous boundary in
the spinor case the situation is different: the adjacent Weyl chamber had no
influence on the values of $M^{\omega _{2}}\left( \mu ,p\right) $ with $\mu
\in \overline{C^{\left( 0\right) }}$ (on the corresponding subdiagonal the
function has zero values). In particular this results in the following
property.

\begin{conjecture}
The ''second diagonal'' of the highest weights for $(L^{\omega
_{1}})^{\otimes p}$ starts with zero: $M^{\omega _{1}}\left( \left(
p-1\right) \omega _{1},p\right) =0.$
\end{conjecture}

\section{Solutions for recurrence relations}

We have found out that the multiplicity functions $M^{\omega _{i}}\left( \mu
,p\right) $ are subject to an infinite system of coupled algebraic equations
with simple and obvious boundary conditions. They can be solved step by step.

For example consider the Bratteli-like diagram for $B_{2}$ vector module $%
L^{\omega _{1}}$ and let $p=1,\ldots ,6$. The maximal number of paths that
connect a point in the $p-1$ slice with the points in the $p$-th one is five.

\begin{figure}[tbh]
\noindent\centering{\ \includegraphics[width=130mm]{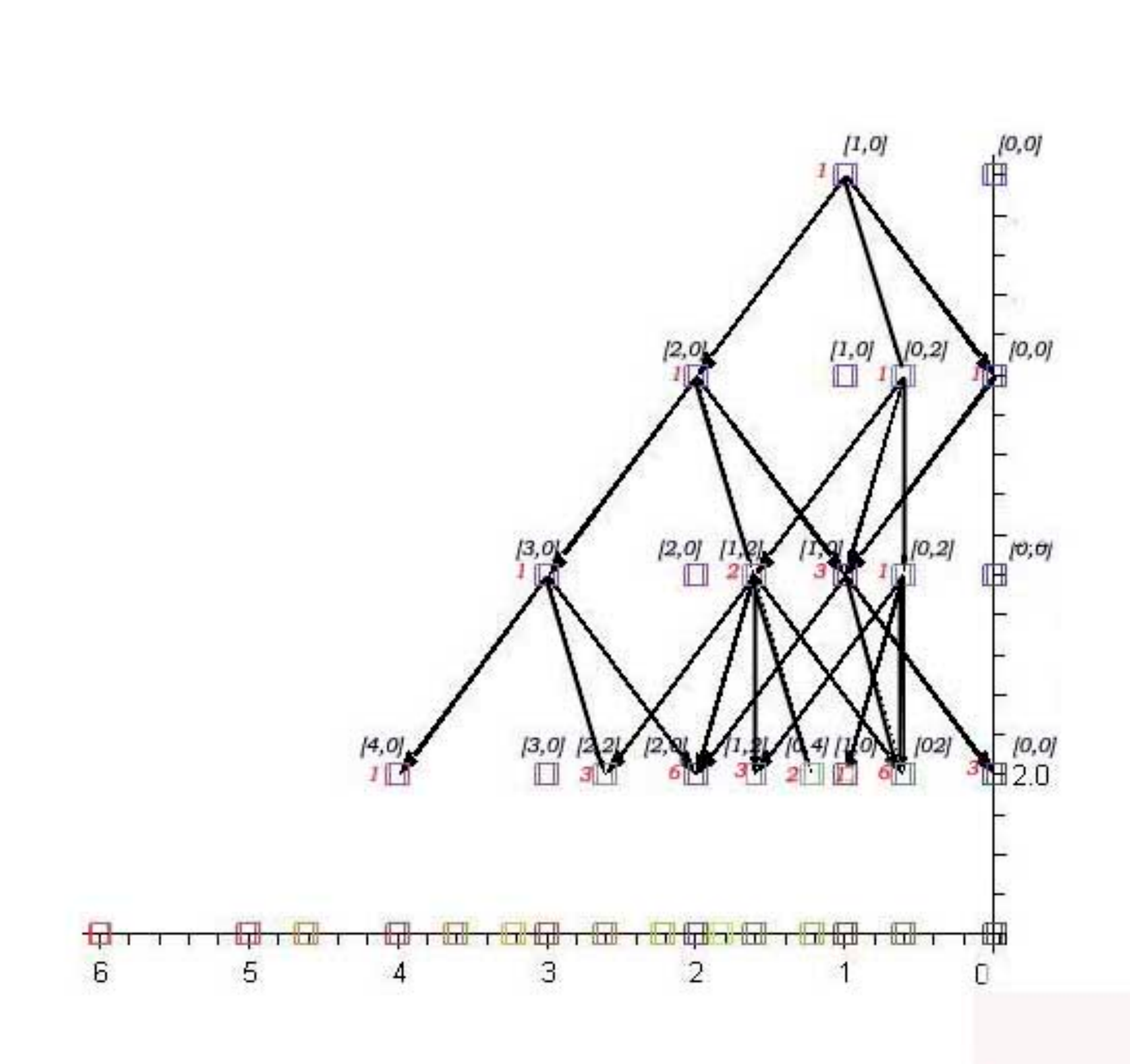} }
\caption{Bratteli-like diagram for $B_{2}$ vector module $L^{\protect\omega
_{1}}$ for $p=1,\ldots ,4$. The maximal number of paths that connect the $p-1
$ slice with the $p$-th one is five. }
\label{bratteli-3}
\end{figure}
Notice that the path counting procedure here is very complicated because of
the boundary effects. Such complexities grow up considerably if we try to
apply that counting procedures to algebras with higher rank. This fact
stimulates special interest to direct studies of the recurrence relations
systems. Moreover if the corresponding equations could be solved this will
give an explicit $p$-dependence of the multiplicity function -- the result
that scaresly could be achieved by combinatorial methods.

We can construct the solution for the recurrence equations successively and
the answer is limited only by the number of equations in the system solved .
This gives the explicit multiplicity dependence on $p$ but for a finite
number of successive weights. Thus having solved the first five equations
for the spinor case we get the following table of functions $M^{\omega
_{2}}\left( \left( a,b\right) ,p\right) $ (here the coordinates $\left(
a,b\right) $ are fixed by the relation $\mu =p\omega _{2}-a\alpha
_{2}-\left( b-1\right) \alpha _{1}$):\\[1mm]

\begin{tabular}{||c||cccccc|}
\hline\hline
$b\setminus a$ & 5/2 & \multicolumn{1}{||c}{2} & \multicolumn{1}{||c}{3/2} &
\multicolumn{1}{||c}{1} & \multicolumn{1}{||c}{1/2} & \multicolumn{1}{||c||}{
0} \\ \hline\hline
1 &  & $\frac{1}{2}p\left( p-3\right) $ &  & $p-1$ &  & $1$ \\ \cline{2-7}
& 0 &  & 0 &  & 0 &  \\ \cline{2-7}
2 &  & $
\begin{array}{c}
\frac{1}{3}\left( p-1\right) \cdot \\
\cdot \left( p+1\right) \left( p-3\right)
\end{array}
$ &  & $\frac{1}{2}p\left( p-1\right) $ &  & 0 \\ \cline{2-7}
& 0 &  & 0 &  & 0 &  \\ \cline{2-7}
3 &  & $
\begin{array}{c}
\frac{1}{12}\left( p-1\right) \left( p-2\right) \cdot \\
\cdot \left( p-3\right) \left( p+2\right)
\end{array}
$ &  & 0 &  &  \\ \hline
\end{tabular}
\\[2mm]

Correspondingly having solved the first fifteen equations for the vector
case we get the following table of functions $M^{\omega _{1}}\left( \left(
a,b\right) ,p\right) $, \\[1mm]

\noindent
\begin{tabular}{||c||c|c|c|c|}
\hline\hline
$b\setminus a$ & 0 & \multicolumn{1}{||c|}{1} & \multicolumn{1}{||c|}{2} &
\multicolumn{1}{||c||}{3} \\ \hline\hline
$p$ & $1$ & 0 & 0 & 0 \\ \hline
$p-1$ & $0$ & $p-1$ & 0 & 0 \\ \hline
$p-2$ & $\frac{1}{2}p\left( p-1\right) $ & $\frac{1}{2}\left( p-1\right)
\left( p-2\right) $ & $\frac{1}{2}p\left( p-3\right) $ & 0 \\ \hline
$p-3$ & ${\
\begin{array}{c}
\frac{1}{6}\left( p-1\right) \\
\left( p-2\right) \left( p-3\right)
\end{array}
}$ & ${\
\begin{array}{c}
\frac{1}{2}p\left( p-1\right) \\
\left( p-3\right)
\end{array}
}$ & ${\
\begin{array}{c}
\frac{1}{3}p\left( p-2\right) \\
\left( p-4\right)
\end{array}
}$ & ${\
\begin{array}{c}
\frac{1}{6}p\left( p-1\right) \\
\left( p-5\right)
\end{array}
}$ \\ \hline
\end{tabular}

\section{Weyl symmetry and solutions for recurrence equations}

When the algebra is simply laced, for example $\frak{g}=A_{n}$, the Weyl
symmetry was proven to be a highly effective tool to solve the set of
recurrences equations for the powers of the first fundamental module $%
L_{A_{n}}^{\omega _{1}}$ \cite{KulLyakhPost2011}. In the simplest case $%
\frak{g}=A_{1}$ the complete set of multiplicity functions for powers of an
arbitrary irreducible module were thus constructed.

In the case of $B_{2}$ the difficulties start when the vector fundamental
module is tensored. The recurrence equation can be solved successively, as
was shown in the previous section, but the complete solution for the
function $M^{\omega _{1}}\left( \left( a,b\right) ,p\right) $ was not found.

Nevertheless the recurrence property (\ref{main rec rel}) permits to
describe the general dependence of the multiplicities on one of the
coordinates. To see this consider the coordinates $(s,t)$ defined by the
relation $\mu =p\omega _{2}-t\alpha _{1}-\left( s-1\right) e_{1}$ . The $t$%
-dependence will be explicitly described, but only for limited values of $%
s=1,2,\ldots $. This description is based on the fact that the Conjecture 7
gives us an explicit answer to the ''first diagonal'' of multiplicities $%
\left\{ m_{\left( 1,t\right) }^{\left( 1\right) p}|t=0,1,\ldots \right\} $.
Starting with this expression and using the relation (\ref{main rec rel})
reformulated for the ''diagonal lines'' of functions we can find explicit
expressions for any such line provided the previous lines are known:
\begin{equation*}
M^{\omega_1 }\left( \left(1,t\right),p\right)=M^{\omega_2 }\left(
\left(\frac{p}{2},\frac{p}{2}-t\right),p\right)=\frac{\Gamma
\left( p+1\right) \left( p+1-2t\right) }{\Gamma \left(
p+2-t\right) \Gamma \left( t+1\right) },
\end{equation*}
\begin{equation*}
M^{\omega_1 }\left( \left(2,t\right),p\right)=\frac{\Gamma \left(
p+1\right) \left( p-t\right) \left( p-2t\right) }{\Gamma \left(
p+2-t\right) \Gamma \left( t\right) \left( t+1\right) },
\end{equation*}
\begin{equation*}
M^{\omega_1 }\left( \left(3,t\right),p\right)=\frac{\Gamma \left(
p+1\right) \left( p-2t-1\right) }{2\Gamma \left( p-t\right) \Gamma
\left( t+1\right) },
\end{equation*}
\begin{eqnarray*}
M^{\omega_1 }\left( \left(4,t\right),p\right) &=&\frac{\Gamma
\left( p+1\right)
\left( p-2t-2\right) }{6\Gamma \left( p+1-t\right) \Gamma \left( t+3\right) }%
\cdot  \\
&&\cdot \left(
\begin{array}{c}
\left( t^{2}+6t+2\right) p^{2}-2\left( t+2\right) ^{2}\left( t+1\right) p+
\\
+t^{4}+4t^{3}+8t^{2}+8t+6
\end{array}
\right) ,
\end{eqnarray*}
\begin{eqnarray*}
M^{\omega_1 }\left( \left(5,t\right),p\right) &=&\frac{\Gamma
\left( p+1\right)
\left( p-2t-3\right) }{24\Gamma \left( p-t\right) \Gamma \left( t+3\right) }%
\cdot  \\
&&\cdot \left(
\begin{array}{c}
\left( t^{2}+11t+6\right) p^{2}-\left( 2t+4\right) \left( t+1\right) \left(
t+2\right) p+ \\
+\left( t^{4}+4t^{3}+8t^{2}+8t+6\right)
\end{array}
\right) ,
\end{eqnarray*}
\begin{eqnarray*}
M^{\omega_1 }\left( \left(6,t\right),p\right) &=&\frac{\Gamma
\left( p+1\right)
\left( p-2t-4\right) }{120\Gamma \left( p-t\right) \Gamma \left( t+4\right) }%
\cdot  \\
&&\cdot \left(
\begin{array}{c}
\left( t^{3}+21t^{2}+86t+36\right) p^{3}- \\
-\left( 3t^{4}+54t^{3}+309t^{2}+654t+276\right) p^{2}+ \\
+\left( 3t^{5}+45t^{4}+326t^{3}+1086t^{2}+1408t+516\right) p- \\
-\left( t+1\right) \left( t+3\right) \left(
t^{4}+8t^{3}+68t^{2}+208t+12\right)
\end{array}
\right) ,
\end{eqnarray*}
and so on. We see that beginning from $M_{B_2 }^{\omega_1 }\left(
\left(4,t\right),p\right)$ only some factor of the multiplicity
function can be presented as a product of simple binomials like
$\left( p-x\right) $. In the forthcoming publications we shall
discuss this property in details.$M^{\omega_1 }\left(
\left(4,b\right),p\right)$

\section{Conclusions}

The tensor powers decomposition algorithm based on singular
weights and injection fan technique was proven to be an effective
tool in multiplicity property studies. Its abilities were
demonstrated on tensor powers decompositions of $B_2$-fundamental
modules. This algorithm is universal and can be applied to
investigate decomposition properties in case of an arbitrary
simple Lie algebra and its arbitrary module.

As it was predicted in \cite{KulLyakhPost2011} in non-simply laced case the
Weyl symmetry properties are insufficient to provide the final solution
for the corresponding set of recurrence relations for multiplicity functions
(at least this appeared to be true for the vector fundamental modules).
Nevertheless (this was shown above in our studies of fundamental $B_2$-modules)
important properties of multiplicity coefficients for any highest weight $\nu$
can be found by constructing the functions
$M^{\omega _{i}}\left( \left( \nu \right) ,p\right) $ successively i.e. by
constructing the solution for a final part of the full set of recurrence relations.

\section{Acknowledgments}
Supported by the Russian Foundation for Fundamental Research grant N 09-01-00504
and the "Dynasty" Foundation.

\end{document}